\documentclass[12pt]{amsart}

\usepackage{amssymb}
\usepackage{txfonts}

\usepackage{enumerate,color}
\usepackage[pagebackref,colorlinks,linkcolor=red,citecolor=blue,urlcolor=blue,hypertexnames=true]{hyperref}
\usepackage{url}

\usepackage{upgreek}

\renewcommand{\subset}{\subseteq}
\newtheorem{theorem}            {Theorem}[section]
\newtheorem{corollary}          [theorem]{Corollary}
\newtheorem{proposition}        [theorem]{Proposition}

\newtheorem{lemma}              [theorem]{Lemma}

\def\cA{ {\mathcal A} }

\def\cK{ {\mathcal K} }

\def\aA{A(\mathbb A)}

\newcommand{\df}[1]{{\bf{#1}}{\index{#1}}}

\newcommand{\hVec}[2]{\begin{pmatrix} #1 & #2\end{pmatrix}}
\newcommand{\vVec}[2]{\begin{pmatrix} #1 \\ #2\end{pmatrix}}

\newcommand{\norm}[1]{\left\lvert#1\right\rvert}
\newcommand{\Norm}[1]{\lVert#1\rVert}
\newcommand{\innerP}[2]{\left\langle #1,#2\right\rangle}

\newcommand{\MBBA}{\mathbb{A}}
\newcommand{\MBBC}{\mathbb{C}}
\newcommand{\MBBD}{\mathbb{D}}
\newcommand{\MBBN}{\mathbb{N}}
\newcommand{\MBBR}{\mathbb{R}}
\newcommand{\MBBT}{\mathbb{T}}
\newcommand{\MBBZ}{\mathbb{Z}}
\newcommand{\MCA}{\mathcal{A}}
\newcommand{\MCE}{\mathcal{E}}
\newcommand{\MCF}{\mathcal{F}}
\newcommand{\MCL}{\mathcal{L}}
\newcommand{\MCM}{\mathcal{M}}
\newcommand{\MCN}{\mathcal{N}}
\newcommand{\MCV}{\mathcal{V}}
\newcommand{\MCW}{\mathcal{W}}
\newcommand{\HINF}{\text{H}^{\infty}}
\newcommand{\LINF}{\text{L}^{\infty}}
\newcommand{\HT}{\text{H}^2}
\newcommand{\LT}{\text{L}^2}

\newcommand{\AD}{\MBBA(\MBBD)}
\newcommand{\PCT}{\mathbb{P}^1(\MBBC)}

\newcommand{\dt}{\;\text{d}t}
\newcommand{\dmu}{\;\text{d}\mu}
\newcommand{\dnu}{\;\text{d}\nu}

\newcommand{\dndm}{\frac{\dnu}{\dmu}}

\DeclareMathOperator{\Span}{span}
\DeclareMathOperator{\Sgn}{sgn}
\DeclareMathOperator{\essinf}{ess\ inf}
\DeclareMathOperator{\esssup}{ess\ sup}
\DeclareMathOperator{\Index}{index}

\makeindex

\begin{document}

\title[Eigenvalues for Toeplitz operators]{Eigenvalues of Toeplitz operators on the Annulus and Neil Algebra}

\author[Broschinski]{Adam Broschinski${}^*$}
\address{Adam Broschinski, Department of Mathematics\\
  University of Florida, Gainesville 
   }
   \email{aebroschinski@ufl.edu}
\thanks{${}^*$ I thank my advisor, Scott McCullough, for his advice and patience.}

\subjclass[2010]{47Axx (Primary).
47B35  (Secondary)}

\date{\today}
\keywords{Toeplitz operator, annulus, Neil parabola, bundle shift}

\begin{abstract}
  By working with all collection of all the Sarason Hilbert Hardy
  spaces for the annulus algebra  an improvement to the results of
  Aryana and Clancey on eigenvalues of self-adjoint Toeplitz
  operators on an annulus is obtained.  The ideas are applied
  to Toeplitz operators on the Neil algebra.  These examples
  may provide a template for a general theory of Toeplitz operators
  with respect to an algebra.
\end{abstract}

\maketitle

\section{Introduction}
  In this article, eigenvalues for 
  self-adjoint Toeplitz operators with real symbols associated to the Neil algebra and the algebra of bounded analytic functions on an annulus are investigated.

  The Neil algebra $\cA$ is the subalgebra of $\HINF$ consisting of those $f$ whose derivative at $0$ is $0$.  Pick interpolation in this, and other related more elaborate subalgebras of $\HINF$, is a current active area of research with \cite{DP}, \cite{DPRS}, \cite{BBtH}, \cite{BH1}\cite{BH2} \cite{JKM}, and \cite{K} among the references.

  For the algebra $\aA$ of functions analytic on the annulus $\MBBA$ and continuous on the closure of $\MBBA$ the results obtained here give finer detail than those of Aryana and Clancey \cite{Ar1}, \cite{AC} (see also \cite{Ar2} and \cite{C}) in their generalization of a result of Abrahamse \cite{Ab}. The proofs are accessible to readers familiar with basic functional analysis and function theory on the annulus as found in either \cite{Fisher} or \cite{Sarason}; in particular, they make no use of theta functions.

  The approach used and structure exposed here applies to many other algebras, including $\HINF(R)$ for a (nice) multiply connected domain in $\MBBC$ and finite codimension subalgebras of $\HINF$, though the details would necessarily be more complicated and less concrete than for the two algebras mentioned above.

  The article proper is organized as follows.
  The algebras $\aA$ and $\MCA$ are treated in Sections \ref{sec:Annulus} and \ref{sec:Neil} respectively.
  These sections can be read independently.
  Only the standard theory of $H^2$ is needed for Section \ref{sec:Neil}.
  The article concludes with Section \ref{sec:Boundaries}; it provides an additional rationale for considering families of representations when studying Toeplitz operators associated to the algebras $\cA$ and $\aA$.

I thank the reviewer for their suggestions which greatly this improved manuscript.
\section{Toeplitz operators on the annulus}
\label{sec:Annulus}
 Fix $0<q<1$ and let $\MBBA$ denote the annulus,
\[
 \MBBA = \MBBA_q = \{z\in\MBBC: q<|z|<1 \}.
\]
  The boundary $B$ of $\MBBA$ has two components
\[
  B_q = \{z\in\MBBC: \norm{z}=q\}
\]
and
\[
  B_1 =\{z\in\MBBC:\norm{z}=1\}.
\]
It is well known that for $\MBBA$ the analog of the classical Hilbert Hardy space  $\HT$ on the disc is a one parameter family of Hilbert spaces that can be described in several different ways \cite{AD}, \cite{Sarason}, or \cite{Ab}.
For our purposes the following is convenient.
Following \cite{Sarason} we will use the universal covering space of the annulus, $\widehat{\MBBA}=\{(r,t)\in\MBBR^2: q<r<1\text{ and }-\infty<t<\infty\}$ with locally conformal coordinates given by the map $\phi(r,t)\mapsto re^{it}$, to define modulus automorphic functions.
A \df{Modulus Automorphic} function, $F$, on $\widehat\MBBA$ is a meromorphic function on $\widehat\MBBA$ that satisfies
\[
\norm{F(r,t)}=\norm{F(r,t+2n\pi)}\text{ for all }q<r<1,\ 0\leq t<2\pi\text{ and }n\in\MBBZ.
\]
So, although $f\coloneqq F\circ\phi^{-1}$ may be multivalued on $\MBBA,$ the function  $\norm{f}$ is single valued.
Because an analytic function is determined, up to a unimodular constant, by it modulus,
if $F$ is modulus automorphic, then
there exists a unimodular constant, $\lambda_{F}$, such that $F(r,t+2\pi)\equiv\lambda_{F} F(r,t)$.
The \df{index} of $F$, denoted by $\Index(F)$, is the unique $\alpha\in[0,1)$ such that $\alpha=(2\pi i)^{-1}\log\lambda_F$.
Let $\mu_j$ denote the multiple of arclength measure on $B_j$ weighted so that $\mu_j(B_j)=2\pi$ and let $\mu=\mu_1+\mu_q$. Given $\alpha\in[0,1),$ define an analog of $\HT(\MBBD)$ in the following way
\[
  \HT_{\alpha}(\MBBA)\coloneqq\{F\circ\phi^{-1}: \Index(F)=\alpha \text{ and } \int_{B}\norm{F\circ\phi^{-1}}^2 \dmu<\infty\}.
\]

In \cite[Section 7]{Sarason} Sarason established the following important properties of $\HT_{\alpha}(\MBBA)$:
\[
	\HT_{\alpha}(\MBBA)\subset\LT(\MBBA) \text{ for all } \alpha\in[0,1)
\]
 and, letting $\chi(z)=z,$
\[
	\HT_{\alpha}(\MBBA)=\{\chi^{\alpha}f\ :\ f\in\HT_0(\MBBA)\}.
\label{zaHT}
\]
Moreover, Sarason showed that the Laurent polynomials are dense in $\HT_0(\MBBA)$ and
thus $\HT_0(\MBBA)$ admits an analog to Fourier Analysis on the disk.

Turning to multiplication and Toeplitz operators on the $\HT_\alpha$ spaces, let $C(\overline{\MBBA})$ denote the Banach algebra (in the uniform  norm) of continuous functions on the closure of $\MBBA$.
The annulus algebra, $\aA,$ is the (Banach) subalgebra of $C(\overline{\MBBA})$ consisting of those $f$ which are analytic in $\MBBA$.
It is easy to see that each $\HT_{\alpha}$ space in invariant for $\aA$ in the sense
  that each  $a\in\aA$ determines a bounded linear operator $M^{\alpha}_a$ on $\HT_{\alpha}$ defined by
\[
M^{\alpha}_a f=af.
\]
Moreover, the mapping $\pi_{\alpha}:\aA\to B(\HT_{\alpha})$ defined by $\pi_{\alpha}(a)=M_a^{\alpha}$ is
 a unital representation of the algebra $\aA$ into the the space $B(\HT_\alpha)$
  of bounded linear operators on the Hilbert space $\HT_\alpha$.

Next let $\phi\in\LINF$ denote a real-valued function on $B$.
The symbol $\phi$ determines a family, one for each $\alpha$, of Toeplitz operators.
Specifically, let $T_{\phi}^{\alpha}$ denote the \df{Toeplitz} operator on $\HT_{\alpha}$ defined by
\[
\HT_{\alpha}\ni f\mapsto P_{\alpha}\phi f,
\]
where $P_{\alpha}$ is the projection of $\LT(B)$ onto $\HT_{\alpha}$.
 A function $g\in\HT_{\alpha}$ is \df{outer} if
 $\{ag\in\HT_{\alpha}:a\in\aA\}$ is dense in
 the Hilbert space $\HT_{\alpha}$ (see \cite[Theorem 14]{Sarason}).

The following is the main result on the existence of eigenvalues for Toeplitz operators on $\MBBA$.
\begin{theorem}
\label{thm:mainAnnulus}
  Fix a real-valued $\phi\in \LINF$. Let $\alpha\in[0,1)$  and a nonzero
  $g\in \HT_{\alpha}$ be given.
  If $T^{\alpha}_{\phi}g=0$, then $g$ is outer and moreover there exists a nonzero $c\in\MBBR$ such that
\begin{equation}
\label{eq:AnnulusMagic}
\phi\norm{g}^2=c\log\norm{\chi q^{-1/2}}.
\end{equation}

If there is an $\alpha$ and an outer function $g\in\HT_\alpha$ such that
  Equation \eqref{eq:AnnulusMagic} holds, then $T^{\alpha}_{\phi}g=0$, where $\alpha$ is necessarily the index of $g$.
 Thus $\alpha$ is congruent modulo 1 to,
\begin{equation}
\label{eq:AnnulusSpace}
  \frac{1}{4\pi\log q}\left(\int_{B_1}\log\norm{\phi}\dmu_1-\int_{B_q}\log\norm{\phi}\dmu_q\right).
\end{equation}
 In particular, there exists at most one $\alpha$ such that $T^{\alpha}_{\phi}$ has eigenvalue $0$ and the dimension of this eigenspace is at most one.
\end{theorem}

 Before we prove Theorem \ref{thm:mainAnnulus}, we pause to collect two corollaries.
 We say that $\lambda$ is an \df{eigenvalue of $\phi$ relative to $\aA$} if there exists an
  $\alpha\in[0,1)$ and nontrivial solution $g$ to $T^{\alpha}_{\phi}g=\lambda g$.
\begin{corollary}
\label{cor:mainAnnulus}
If
\[
\esssup\{\phi(z): z\in B_q\}=m<0<M=\essinf\{\phi(z): z\in B_1\}
\]
or
\[
\esssup\{\phi(z): z\in B_1\}=m<0<M=\essinf\{\phi(z): z\in B_q\},
\]
then each $\lambda\in(m,M)$ is an eigenvalue of $\phi$ relative to $\aA$, the
latter case only happening when the $c$ from Theorem \ref{thm:mainAnnulus} is negative.
Further, $M$ (resp. $m$) is an eigenvalue if and only if
$\frac{\log\norm{\chi q^{-1/2}}}{\phi-M}\in\text{L}^1$ (resp. $\frac{\log\norm{\chi q^{-1/2}})}{\phi-m}\in\text{L}^1$).
\end{corollary}

\begin{corollary}
\label{cor:intAnnulus}
The set of eigenvalues of $\phi$ relative to $\aA$ is either empty, a point, or an interval.
\end{corollary}

 The following Corollary of Theorem \ref{thm:mainAnnulus} and Corollary \ref{cor:mainAnnulus}
 generalizes the main result of \cite{AC} for the annulus.

\begin{corollary}
 \label{cor:AC}
   With the hypotheses of Corollary \ref{cor:mainAnnulus},  if either
\[
  \int_B \log|\phi-M| \dmu= - \infty
\]
  or
\[
  \int_B \log|\phi-m| \dmu= - \infty,
\]
  then for each $\alpha$ the Toeplitz operator $T_{\phi}^{\alpha}$ has infinitely many eigenvalues
  in the interval $(-m,M)$.
\end{corollary}

The remainder of this section is organized as follows. Subsection \ref{subsec:proofMainAnnulus}
contains the proof of Theorem \ref{thm:mainAnnulus}.  The corollaries are proved in Subsection
\ref{subsec:proofCorMainAnnulus}.

\subsection{Proof of Theorem \ref{thm:mainAnnulus}.}
\label{subsec:proofMainAnnulus}
Let
\[
\aA^* \coloneqq \{f^*:f\in\aA\}.
\] 
In the context of Theorem \ref{thm:mainAnnulus}, suppose $T^{\alpha}_{\phi}g=0$.
Using the fact that, if $a\in\aA$, then $ag\in\HT_{\alpha}$ it follows that
\begin{align*}
0=&\innerP{T^{\alpha}_{\phi}g}{ag}\\
=&\int_{B}\phi\norm{g}^2 a^* \dmu.
\end{align*}
Since $\phi\norm{g}^2$ is real-valued
\[
\int_{B}\phi\norm{g}^2 a \dmu=0
\]
 too.
Thus $\phi\norm{g}^2$ annihilates $\aA\oplus\aA^*$.
In fact we know the following about measures that annihilate $\aA\oplus\aA^*$.
\begin{proposition}
\label{prop:Aperp}
	A measure $\nu<<\mu$ who's Radon-Nikodym derivative with respect to $\mu$ is in $\text{L}^1(B)$ annihilates $\aA\oplus\aA^*$ if and only if there exists a $c\in\MBBC$ such that
\[
	\dndm=c\log\left(\chi q^{-\frac{1}{2}}\right).
\]
\end{proposition}
\begin{proof}
 If $\nu$ annihilates $\aA\oplus\aA^*$, then for each $n\in\MBBZ$
\begin{align*}
  \int_B  \chi^n \dnu &= \int_B \chi^n \dndm \dmu = \sum_{j=0}^1 \int_0^{2\pi} \dndm(q^je^{it}) q^j e^{int} \dt \text{ and}\\
  \int_B \overline{\chi}^n \dnu &= \int_B \overline{\chi}^n \dndm \dmu = \sum_{j=0}^1 \int_0^{2\pi} \dndm(q^je^{it}) q^j e^{-int} \dt.
\end{align*}
Hence
\[
\begin{aligned}
  \int_0^{2\pi} \dndm(e^{it}) e^{int} \dt &=-\int_0^{2\pi} \dndm(qe^{it}) q^n e^{int} \dt \text{ and}\\
  \int_0^{2\pi} \dndm(e^{it}) e^{-int}\dt &=-\int_0^{2\pi} \dndm(q e^{it}) q^n e^{-int} \dt.
\end{aligned}
\]
By replacing $n$ with $-n$ in the last equation we can see that
\[
  q^n\int_0^{2\pi} \dndm(qe^{it}) e^{int} \dt = q^{-n}\int_0^{2\pi} \dndm(qe^{it}) e^{int} \dt.
\]
  Hence for all $n\in\MBBZ$
\[
  \left(q^n-q^{-n}\right)\int_0^{2\pi} \dndm(q^j e^{it}) e^{int} \dt = 0.
\]
That means that for all $0\neq m\in\MBBZ$ and $j=0,1$ 
\[
  \int_0^{2\pi} \dndm(q^je^{it})e^{imt} \dt =0.
\]
So $\dndm$ must be equal to a constant, $a$, almost everywhere on each boundary with $a\coloneqq\left.\dndm\right|_{B_1}=-\left.\dndm\right|_{B_q}$.
If we choose $c\in\MBBR$ to be $\frac{a}{\log q^{-\frac{1}{2}}}$, then $\dndm=c\log\norm{zq^{-\frac{1}{2}}}$ almost everywhere on $B$.

Now let $\nu$ be a measure such that $\dndm=c\log\norm{zq^{-\frac{1}{2}}}$ almost everywhere on $B$.
Since $\Span\{z^n, \overline{z}^n\ |\ n\in\MBBZ\}$ is dense in $\aA\oplus\aA^*$ it will suffice to show that for $n\in\MBBZ$
\begin{align*}
	\int_B \chi^n\dnu&=\int_B \chi^n c\log\norm{\chi q^{-\frac{1}{2}}}\dmu=0 \text{ and}\\
	\int_B \overline{\chi}^n\dnu &= \int_B \overline{\chi}^nc\log\norm{\chi q^{-\frac{1}{2}}}\dmu=0
\end{align*}
to prove that $\nu$ annihilates $\aA\oplus\aA^*$.
If $n=0$, then 
\begin{align*}
	\int_B c\log\norm{\chi q^{-\frac{1}{2}}})\, \dmu &= \int_{B_1}c\log\left(q^{-\frac{1}{2}}\right)\dmu_1+\int_{B_q}c\log\left(q^{\frac{1}{2}}\right)\, \dmu_q=0.
\end{align*}
If $n\neq 0$, then for $j=1,q$, 
\[
    \int_{B_j} c\log\norm{\chi q^{-\frac12}} = c\log\left(j q^{-\frac12}\right)  \int_{B_j} \chi^n \dmu_1 = 0. 
\]
A similar computation shows that $\int_B \overline{\chi}^n c\log\norm{\chi q^{-\frac{1}{2}}}\dmu=0$.
\end{proof}
Combining the fact that if $T_\phi^\alpha g=0,$ then $\phi\norm{g}^2$ annihilates $\aA\oplus\aA^*$ and the above proposition we see that if $T_\phi^{\alpha} g=0,$ then there exists some $c\in\MBBR$ such that $\phi\norm{g}^2=c\log\norm{\chi q^{-\frac{1}{2}}}$.

The next objective is to show that $g$ is outer.
First we need the following definition. A function $\theta$ in $\HT_{\beta}$ is \df{inner} if $\norm{\theta}=1$ on $B$.
Sarason in \cite[Theorem 7]{Sarason} proved a version of inner-outer factorization for the annulus: given $f\in \HT_\alpha$, there is a $\beta$ and an inner function $\psi \in \HT_\beta$ and outer function $F\in \HT_{\alpha-\beta}$ such that
\begin{equation}
 \label{eq:IO-annulus}
  f=\psi F.
\end{equation}
Let $g=\psi F$ denote the inner-outer factorization of $g$ as an $\HT_{\alpha}$ function as in Equation \eqref{eq:IO-annulus} and let $\beta\in[0,1)$ be the index of $\psi$.
Since $\Index(\psi)+\Index(F)=\Index(g),$ we have that $\chi^{\beta}F\in \HT_{\alpha}$.
Similarly we have that $C\coloneqq \chi^{-\beta}\psi\in\HT$ which means that its restrictions to each of 
 the boundary components $B_1$ and $B_q$ is representable as a Fourier series 
 whose coefficients  we will denote by $\widehat{C}_q(n)$ and $\widehat{C}_1(n)$ respectively.  Moreover by \cite[Lemma 1.1]{Sarason} we know that $\widehat{C}_1(n)=q^{-n}\widehat{C}_{q}(n)$.
Since we showed above that $\phi\norm{g}^2=c\log\norm{\chi q^{-\frac{1}{2}}}$ for some $c\in\MBBR$, for any $n\in\MBBZ$
\begin{equation}
\label{eq:FourierAnnulus}
  \begin{split}
    0&=\innerP{T^{\alpha}_{\phi}g}{\chi^n\chi^{\beta}F}\\
     &=\int_{B}\phi\norm{g}^2 \norm{\chi}^{2\beta}\chi^{-\beta} \psi\overline{\chi}^n \dmu\\
     &=\int_{B}c\log\left(\norm{\chi q^{-1/2}}\right)\norm{\chi}^{2\beta}C\overline{\chi}^n \dmu\\
     &=c\log(q^{1/2})\left(q^{n+2\beta}\widehat{C}_q(n)-\widehat{C}_1(n)\right)\\
     &=c\log(q)\widehat{C}_q(n)\left(q^{n+2\beta}-q^{-n}\right).
  \end{split}
\end{equation}

From Equation \eqref{eq:FourierAnnulus} it follows that, for each $n$,
  either $\widehat{C}_q(n)=0$ or $n + \beta =0$.
Since $\beta\in[0,1)$ and $\widehat{C}_q(m)\neq 0$ for some $m$, it follows that $\widehat{C}_q(n)=0$ for $n\neq 0$ and $\beta= 0$.
Thus $\psi$ is a unitary constant and $g$ is outer.

Next assume that $g\in\HT_{\alpha}$ is outer and equation \eqref{eq:AnnulusMagic} holds. By Proposition \ref{prop:Aperp}, 
\[
	\innerP{T_\phi^\alpha g}{ag}=\int_B \phi\norm{g}^2\overline{a}\dmu=\int_B c\log\norm{\chi q^{-\frac{1}{2}}} \overline{a}=0,
\]
 for every $a\in \AD.$
 Further since $g$ is outer $\{ag\ |\ a\in\aA\}$ is dense in $\HT_\alpha$. Thus $T_\phi^\alpha g=0$ which proves the second part.

To prove the third part of the Theorem, simply choose $\alpha$ to be the index of $g$.
From \cite[Theorem 6]{Sarason} we know that, modulo one, the index of $g$ is
\[
\frac{-1}{2\pi\log q}\left(\int_{B_1}\log\norm{g}\dmu_1-\int_{B_q}\log\norm{g}\dmu_q\right).
\]
Applying the fact that $\norm{g}=\left(\frac{c\log\norm{\chi q^{-1/2}}}{\phi}\right)^{1/2}$
  the above expression simplifies to
\[
\frac{1}{4\pi\log q}\left(\int_{B_1}\log\norm{\phi}\dmu_1-\int_{B_q}\log\norm{\phi}\dmu_q\right).
\]

Finally, suppose that $T^{\alpha}_{\phi}g=0$ and also $T^{\beta}_{\phi}h=0$. From what has already been proved $g$ and $h$ are outer and there exists nonzero $c,d\in\MBBR$ such that
\[
\phi\norm{g}^2=c\log\norm{\chi q^{-1/2}}, \ \ \phi\norm{h}^2=d\log\norm{\chi q^{-1/2}}
\]
on $B$.
Since $\phi$ is almost everywhere nonzero we have that $\norm{g}^2=\frac{c}{d}\norm{h}^2$ and because $g$ and $h$ are outer, they are equal up to a complex scalar multiple, see \cite[Theorem 7.9]{Sarason}.

\subsection{Proofs of the corollaries}
\label{subsec:proofCorMainAnnulus}
 To prove Corollary \ref{cor:mainAnnulus}, observe that the first (resp. second) displayed inequality implies, for $m<\lambda<M,$ that
\begin{equation}
 \label{eq:psilambda}
\psi=\frac{\log\norm{\chi q^{-1/2}}}{\phi-\lambda}
\end{equation}
takes nonnegative (resp. nonpositive) values and is essentially bounded above (resp. below) and below (resp. above) away from zero.
 Hence by \cite[Theorem 9]{Sarason} there exists an outer function $g$ such that such that $|g|^2=\psi$ (resp. $|g|^2 = -\psi$). 
From Theorem \ref{thm:mainAnnulus} there is a $\alpha$ such that $T^{\alpha}_{\phi}g=\lambda g$.
 The case $\lambda=M$ (resp. $\lambda=m$) is similar, but now, while $\psi$ ‎is still essentially bounded below away from zero, it need not be integrable.
If $\psi$ is integrable, than the argument above shows it is an eigenvector with eigenvalue $M$ (resp. $m$).
On the other hand, if $M$ (resp. $m$) is an eigenvalue, then there is an outer function $g$ so that $\psi = |g|^2$ and hence $\psi$ is integrable.

It suffices to prove Corollary \ref{cor:intAnnulus} for $M=\essinf\{\phi(z): z\in B_1\}$ and $m=\esssup\{\phi(z): z\in B_q\}$.
By Corollary \ref{cor:mainAnnulus} if $m<M$ then the set of $\phi$ relative to $\aA$ contains the interval $(m,M)$ and otherwise \ref{thm:mainAnnulus} says that the set of eigenvalues is at most a point.
So we must show that if $\lambda>M$ or $\lambda<m$ then $\lambda$ is not an eigenvalue.
To this end, suppose $\lambda>M$.
Since $\lambda>M=\essinf\{\phi(z): z\in B_1\}$ we have that $\mu\left(\{\phi(z)-\lambda<0:z\in B_1\}\right)>0$.
On the other hand, if $\lambda$ is an eigenvalue, then there is a nonzero $d\in\MBBR$ and multivalued outer function $h\in \LT(B)$ such that $(\phi-\lambda)\norm{h}^2=d\log\norm{\chi q^{-1/2}}$ which  implies that either $\left.(\phi-\lambda)\right|_{B_q}$ is positive almost everywhere or $\left.(\phi-\lambda)\right|_{B_1}$ is positive almost everywhere.
This is a contradiction since we know that $(\phi-\lambda)$ is negative on the inner boundary of the annulus and not positive almost everywhere on the outer boundary.
This proves Corollary \ref{cor:intAnnulus} for $\lambda>M$.
The proof for the case $\lambda<m$ proceeds analogously. The details are omitted.

It suffices to prove Corollary \ref{cor:AC} when $M=\essinf\{\phi(z):z\in B_1\}$ and $m=\esssup\{\phi(z):z\in B_q\}$.
Assume that $\int_B\log\norm{\phi-M}\dmu=-\infty$.
Given $m<\lambda <M$, by Corollary \ref{cor:mainAnnulus}, there is a unique, up to scalar multiple, outer function $g_\lambda$ of $T_{\phi-\lambda}$ whose modulus squared is given by $\psi$ in equation \eqref{eq:psilambda}.
By Theorem \ref{thm:mainAnnulus}, the index $\alpha_\lambda\in [0,1)$ of $g_\lambda$ is congruent, modulo one, to 
\[
\beta_{\lambda}\coloneqq\frac{1}{4\pi\log q}\left(\int_{B_1}\log\norm{\phi-\lambda}\dmu_1-\int_{B_q}\log\norm{\phi-\lambda}\dmu_q\right).
\]
 Notice that as $\lambda$ approaches $M$ on $B_q$ we have that $\int_{B_q}\log\norm{\phi-\lambda}<\infty$ since $\phi\in\LINF$ and $\esssup\{\phi(z):z\in B_q\}<M$.
It follows from the monotone convergence theorem that $\beta_{\lambda}$ approaches $-\infty$ as $\lambda$ approaches $M$.
Hence $\alpha_\lambda$ takes every value in the interval $[r,1)$ infintely often for every choice of $0\le r<1$. 
Which completes the proof in the case $M=\essinf\{\phi(z):z\in B_1\}$ and $\int_B\log\norm{\phi-M}\dmu=-\infty$.
 The the proof for the case $m=\esssup\{\phi(z):z\in B_q\}$ and $\int_B\log\norm{\phi-m}\dmu=-\infty$ proceeds similarly.

\section{Toeplitz operators on the Neil parabola}
\label{sec:Neil}

Let $\MCA$ denote the Neil Algebra; i.e.,
 $\MCA$ is the unital subalgebra of the disc algebra $\MBBA(\MBBD)$ consisting of those $f$ with $f^\prime(0)=0$.
Each subspace $\MCV \subset \MBBC \oplus \MBBC z$ determines a subspace
\[
 \HT_\MCV = \HT \ominus \MCV
\]
 of the classical Hardy space $\HT$ which is invariant for $\cA$ in the following sense.
Each $a\in \MCA$ determines a bounded linear operator $M_a^{\MCV}$ on $\HT_\MCV$ defined by
\[
  M_a^\MCV f = af.
\]
Moreover, the mapping $\pi_\MCV:\cA \to B(\HT_\MCV)$ defined by $\pi_\MCV(a) = M_a^\MCV$ is a unital representation.
Here $B(\HT_\MCV)$ is the algebra of bounded operators on $\HT_\MCV$.
A further discussion of the collection of representations $\pi_\MCV$ can be found in Section \ref{sec:Boundaries}.

Let $\phi$ denote a real-valued function on the unit circle $\MBBT$.
The symbol $\phi$ determines a family, one for each $\MCV$, of Toeplitz operators.
Specifically, let $T_\phi^\MCV$ denote the \df{Toeplitz} operator on $\HT_\MCV$ defined by
\[
  \HT_\MCV \ni f \mapsto  P_\MCV \phi f,
\]
where $P_\MCV$ is the projection of $\LT(\MBBT)$ onto $\HT_\MCV$.

Letting $\chi(z)=z$ we get the following as the main result on the existence of eigenvalues for Toeplitz operators on $\HT_\MCV$.

\begin{theorem}
 \label{thm:mainNeil}
   Fix a real-valued $\phi \in \LINF$ and let $\MCV\subset \MBBC \oplus \MBBC z$ and nonzero  $g\in \HT_\MCV$ be given.
   If $T_\phi^\MCV g=0$, then $g$ is outer and moreover there is a $c\in\MBBC$ such that, on $\MBBT$,
\begin{equation}
 \label{eq:NeilMagic}
   \phi \norm{g}^2 = c\chi+(c\chi)^*.
\end{equation}

 Conversely, if there is a $c\in\MBBC$ and outer function $g\in \HT$ such that Equation \eqref{eq:NeilMagic} holds, then $T_\phi^\MCV g=0$, where $\MCV$ is uniquely determined by the values $g(0)$ and $g^\prime(0)$.

 In particular, there exists at most one $\MCV$ such that $T_\phi^\MCV$ has eigenvalue $0$ and the dimension of this eigenspace is at most one.
\end{theorem}

 Before turning to the proof of Theorem \ref{thm:mainNeil}, we pause to state the analogs of corollaries \ref{cor:mainAnnulus}, and \ref{cor:intAnnulus}.
 By analogy with the case of the annulus, we  say that $\lambda$ is an
\df{eigenvalue of $\phi$ relative to $\MCA$} if there exists a $\MCV \subset \MBBC \oplus \MBBC z$ and nontrivial solution to $T_\phi^\MCV g=\lambda g$.

\begin{corollary}
 \label{cor:mainNeil}
    If there is a $c\in\mathbb C$ such that 
 \[
  \esssup\{ \phi(z): cz+(cz)^* <0\} = m < 0 < M = \essinf\{\phi(z): cz+(cz)^* >0\},
 \]
  then each $\lambda\in (m,M)$  is an eigenvalue of $\phi$ relative to $\cA$.
  Further, for each such $\lambda$ there is an essentially unique outer function $f_\lambda$ such that
\[
  (\phi-\lambda) |f_\lambda |^2 = c\chi + (c\chi)^*.
\]
  (Here $c$ is independent of $\lambda$.)

   Moreover, $M$ (resp. $m$) is an eigenvalue if and only if $\frac{c\chi+(c\chi)^*}{\phi-M}$
  (resp. $\frac{c\chi+(c\chi)^*}{\phi-m}$) is in $\text{L}^1.$
\end{corollary}

\begin{corollary}
 \label{cor:intNeil}
   The set of eigenvalues of $\phi$ relative to $\cA$ is either empty, a point,
  or an interval.
\end{corollary}

If we are given a $\phi$ that satisfies Corollary \ref{cor:mainNeil} and a $\lambda\in(m,M)$ the following corollary allows us to determine exactly what $\HT_\MCV$ space the resulting outer function $f_\lambda$ is in.
But first for $z\in\MBBD$ and $t$ real, let
\[
  H(z,t) =\frac{e^{it}+z}{e^{it}-z}.
\]

\begin{corollary}
  \label{cor:lambdatov}
Under the hypotheses of corollary \ref{cor:mainNeil} let
\[
h_c(z)\coloneqq \exp\left(\frac{1}{2\pi} \int_{-\pi}^\pi H(z,t) \log\norm{ce^{it}+c^*e^{-it}}^{1/2} \dt\right),
\]
and
\[
g_{\lambda}(z) \coloneqq \exp\left(\frac{1}{2\pi}\int_{-\pi}^\pi H(z,t)\log\norm{\phi(t)-\lambda}^{-1/2} \dt \right).
\]
The eigenvector $f_{\lambda}$ of $T_{\phi}^\MCV$ associated with the eigenvalue $\lambda$ is in the $\HT_\MCV$ space
  where (nontrivial) $\MCV\subset \MBBC \oplus \MBBC z$ is orthogonal to
\[
h_c(0) {g}_{\lambda}(0)+\left(h_c(0) g_{\lambda}^\prime(0)+ h^{\prime}_c(0) g_{\lambda}(0)\right)z
\]
  Moreover since $h_c$ and $g_{\lambda}$ are outer functions neither $h_c(0)$ or $g_{\lambda}(0)$ are zero.
\end{corollary}

There is no analog of Corollary  \ref{cor:AC} or of the main result of \cite{AC} for the Neil parabola.
  In fact Corollary \ref{cor:lambdatov} implies that if $e\subset \MBBC \oplus \MBBC z$ is spanned by $1,$ then no Toeplitz operator on $\HT_e$ has eigenvalues.
Although  an easier way to see that no Toeplitz operator on $\HT_e$ has eigenvalues is to note that $\HT_e=z\HT$.
Moreover for similar reasons no Toeplitz operator on $\HT_\MCV$ has eigenvalues if $\MCV=\{0\}$ or $\MCV = \MBBC \oplus \MBBC z$.
  In fact the following corollary says that there are many nonzero proper $\MCV$, not just $e$, such that $T^\MCV_\phi$ has no eigenvalues.
  To prove this we identify each nonzero proper $\MCV$ with an element of the complex projective line, $\PCT$, which is
  $X=\MBBC^2\setminus \{0\}$ modulo the equivalence relation
  $v\sim w$ if and only if there is a complex number $\lambda$ such
  that $v=\lambda w$.  Let $\pi:X \to \PCT$ denote the quotient mapping of $v\in\MBBC^2\setminus\{0\}$.
	The space $\PCT$ can be realized as a Riemann surface by the charts $\Phi_j : \MBBC\to \PCT$ defined by $\Phi_0(\zeta) = \hVec{\zeta}{1}^T$ and $\Phi_1(\upxi)=\hVec{1}{\upxi}^T$.
 Indeed, the transition mappings between these charts are $\zeta=\frac{1}{\upxi}$ and $\upxi=\frac{1}{\zeta}$.
 A map $F:\MBBR\to\PCT$ is differentiable if the maps $\Phi_0^{-1} \circ F$ and $\Phi_1^{-1} \circ F$ are differentiable where defined.
 Finally, if $I$ is an interval in $\MBBR$ and $g:I\to X$ is twice differentiable, then so is $\pi \circ g$ and in this case the Hausdorff dimension of the range of $\pi\circ g$ is at most one and in this sense the range is a relatively small subset of $\PCT$. For a discussion of properties of the Hausdorff dimension see \cite{Sc}.

\begin{corollary}
  \label{cor:notAC}
  The function, $\Lambda$, from $(m,M)$ to $\PCT$ defined by
\[
\Lambda:\lambda\mapsto\pi\left(\vVec{h_c(0) g_{\lambda}(0)}{h^{\prime}_c(0)g_{\lambda}(0)+ h_c(0) g^{\prime}_{\lambda}(0)}\right)
\]
is locally Lipschitz with respect to $\lambda$ on $(m,M)$.
Thus in addition to $\MCV=\Span\{1\}$ there exist nonzero proper $\MCV\subset \MBBC \oplus \MBBC z$ such that $T^\MCV_\phi$ has no eigenvalues.
\end{corollary}

The remainder of this section is organized as follows.
Subsection \ref{subsec:proofmainNeil} contains the proof of Theorem \ref{thm:mainNeil}.
The corollaries are proved in Subsection \ref{subsec:proofcormainNeil}.

\subsection{Proof of Theorem \ref{thm:mainNeil}}
\label{subsec:proofmainNeil}
 Let $\mu$ denote normalized arclength measure on $\MBBT$ and let
\[
 \cA^* \coloneqq \{f^*: f\in\cA\}.
\]
 In the context of Theorem \ref{thm:mainNeil}, suppose $T_\phi^\MCV g =0$.
 Using the fact that, if $a\in\cA$, then $ag \in \HT_\MCV$ it follows that
 \begin{align*}
 0 = & \innerP{T_\phi^\MCV\ g}{ag}\\
   = & \int_\MBBT \phi \norm{g}^2 \overline{a} \dmu.
 \end{align*}
Since $\phi\norm{g}^2$ is real-valued it is also the case that
\[
\int_\MBBT\phi\norm{g}^2 a \dmu=0.
\]
Thus $\phi\norm{g}^2$ annihilates $\MCA\oplus\MCA^*$. In fact we know the following about measures that annihilate $\MCA\oplus\MCA^*$.
\begin{proposition}
A measure $\nu<<\mu$ who's Radon-Nikodym derivative with respect to $\mu$ is in $\text{L}^1(\MBBT)$ annihilates $\MCA\oplus\MCA^*$ if and only if there exists a $c\in\mathbb C$ such that 
\[
	\dndm=c\chi+(c\chi)^*.
\]
\end{proposition}
\begin{proof}
Assume that $\nu$ annihilates $\MCA\oplus\MCA^*$, then for each $n\in\MBBZ\setminus\{1,-1\}$
\[
	\int_\MBBT \chi^n \dnu=\int_\MBBT \chi^n\dndm \dmu=0
\]
Hence for any $n$ different from $\pm 1$ the corresponding Fourier coefficent of $\dndm$ is zero.
Thus there exists $c\in\MBBC$ such that all the Fourier coefficients of $\dndm - c\chi + (c\chi)^*$ are $0$, which implies that $\dndm=c\chi+(c\chi)^*$ almost everywhere on $\MBBT$.

Now let $\nu$ be a measure such that $\dndm=c\chi+(c\chi)^*$ almost everywhere on $\MBBT$.
Since the span of the set $\left\{z^n\ |\ n\in\MBBZ\setminus\{1,-1\}\right\}$ is dense in $\MCA\oplus\MCA^*$ it will suffice to show that for $n\in\MBBZ\setminus\{1,-1\}$
\[
	\int_\MBBT \chi^n\dnu=\int_\MBBT\chi^n\left(c\chi+(c\chi)^*\right)\dmu=0
\]
to prove that $\nu$ annihilates $\MCA\oplus\MCA^*$. 
But for $n\in\MBBZ\setminus\{1,-1\}$
\[
	\int_\MBBT\chi^n\left(c\chi+(c\chi)^*\right)\dmu=\int_\MBBT c\chi^{n+1}+\left(c\chi^{n-1}\right)^*=0.
\]
\end{proof}
Combining the fact that if $T_\phi^\MCV g=0,$ then $\phi\norm{g}^2$ annihilates $\MCA\oplus\MCA^*$ and the above proposition we see that if $T_\phi^\MCV g=0,$ then there exists some $c\in\MBBC$ such that $\phi\norm{g}^2=c\chi+(c\chi)^*$.

  The next objective is to show $g$ is outer.
  To this end, let $g=\Psi F$ denote the inner-outer factorization of $g$ as an $\HT$ function.
  Observe that, $z^n F\in \HT_\MCV$ for  integers $n\ge 2$.
  Thus, for such $n$,
\[
 \begin{split}
   0 = & \innerP{T_\phi^v g}{F z^n} \\
     = & \int_\MBBT \phi |g|^2 \Psi \overline{z}^n \dmu\\
     = & \innerP{(cz+(cz)^*)\Psi}{z^n}_{\LT}.
 \end{split}
\]
 It follows, writing $\Psi=\sum_{k=0}^\infty \Psi_k z^k$, that
\[
  c\Psi_{n-1}+c^* \Psi_{n+1}=0
\]
 for $n\ge 2$.  In particular,
\[
  \Psi_{2k+1} = (-\frac{c}{c^*})^k \Psi_1
\]
 for $k\ge 1$ and likewise,
\[
 \Psi_{2k+2} = (-\frac{c}{c^*})^k \Psi_2.
\]
 Because $\Psi\in \HT$ these last two equations imply that
  $\Psi_k=0$ for $k\ge 1$; i.e.,  $\Psi$ is a unimodular constant
 and thus $g$ is outer, and the first part of the
  Theorem is established.

  The proof of the converse uses the following lemma.
\begin{lemma}
 \label{lem:outerdense}
  Given a nonzero $\MCV\subsetneqq \MBBC \oplus \MBBC z$, if $g$ is outer and in $H^2_\MCV$, then the set
\[
 \{ag: a\in\cA\}
\]
 is dense in $\HT_\MCV$.
\end{lemma}

\begin{proof}
 Let $f\in\HT_\MCV$ be given.
Since $g$ is outer there exists a sequence of functions $\{a_n\}\subset\HINF(\MBBD)$ such that $a_n g$ converges to $f$ in $\HT(\MBBD)$.
Let $b_n\coloneqq a_n-a_n^{\prime}(0)$, so each $b_n\in\cA$.
To show that $b_n g$ also converges to $f$ it suffices to show that $a_n^{\prime}(0)$ converges to zero.
To do this note that for each nonzero proper $\MCV$ there exists a pair $(\alpha, \beta)\in\MBBC^2$ not both zero such that if $h\in\HT_\MCV$, then there exists a $\zeta\in\MBBC$ and $q\in\HT(\MBBD)$ such that $h=\zeta\alpha +\zeta\beta z+z^2q$.
Since $g(0)\neq 0$ this means that
\[
\frac{f(0)}{g(0)}\cdot g^{\prime}(0)=f^{\prime}(0).
\]
Because $a_n(0)g(0)$ converges to $f(0)$ and $(a_n g)^{\prime}(0)$ converges to $f'(0)$ we have that
\begin{align*}
\lim_{n\to\infty} a_n(0) &= \frac{f(0)}{g(0)} \\
\lim_{n\to\infty}a_n^{\prime}(0) &=\frac{f^{\prime}(0)-\frac{f(0)}{g(0)}\cdot g^{\prime}(0)}{g(0)}=0
\end{align*}
%
\end{proof}

  To prove the converse, suppose that $g$ is outer and there is a $c\in\mathbb C$ such that
  Equation \eqref{eq:NeilMagic} holds.
Let $\MCV$ be the nonzero subspace of $\MBBC \oplus \MBBC z$ such that $g(0)+g^{\prime}(0)z\in\MCV^{\perp}.$ 
 In particular, $g\in\HT_\MCV$.
It follows that, for any $a\in\cA$,
\[
  \langle T_\phi^\MCV g, ag\rangle
    = \int_\MBBT \phi \norm{g}^2 a^* \dmu=0
\]
 and thus, in view of Lemma \ref{lem:outerdense}, $T_\phi^\MCV g =0$.

 Finally, suppose that $T_\phi^\MCV g = 0$ and $T_\phi^\MCW h=0$. From what has
 already been proved $g$ and $h$ are outer and there exists $c,d\in\MBBC$ such that
\[
   \phi \norm{g}^2 = c\chi +(c\chi)^*, \ \  \phi \norm{h}^2 = d\chi+(d\chi)^*
\]
 on $\mathbb T$.  It follows that $\phi$ is positive almost everywhere
 both where $c\chi+(c\chi)^*$ and $(d\chi)+(d\chi)^*$ are positive. Hence
  $c=td$ for some positive real number $t$. But then, $t \norm{g}=\norm{h}$
  and because $g$ and $h$ are outer, they are equal up to a (complex) scalar
  multiple.

\subsection{Proofs of the corollaries}
\label{subsec:proofcormainNeil}
   To prove Corollary \ref{cor:mainNeil}, observe that the hypotheses imply, for $m<\lambda<M,$ that
\[
  \psi = \frac{c\chi+(c\chi)^*}{\phi-\lambda}
\]
 is  nonnegative, in $\text{L}^1$ and moreover
\begin{equation}
 \label{eq:logint}
  \int_\MBBT \log\norm{\psi}\dmu > -\infty
\end{equation}
 because the same is true with $\psi$ replaced by $c\chi+(c\chi)^*$, $\phi$ is essentially bounded and $\Sgn\left(c\chi+(c\chi)^*\right)=\Sgn(\phi)$.
 Hence there is an outer function $g\in \HT$ such that
\[
  (\phi-\lambda)\norm{g}^2 =c\chi+(c\chi)^*.
\]
  From Theorem \ref{thm:mainNeil} there is a nonzero proper $\MCV$ such that
\[
  T_{\phi}^\MCV g = \lambda g.
\]

The case $\lambda=M$ (resp. $\lambda=m$) are similar, with the only issue being that a hypothesis is needed to guarantee that $\psi$, as defined above, is integrable.

  Turning to the proof of Corollary \ref{cor:intNeil}, because Corollary \ref{cor:mainNeil} implies the interval $(m,M)$ is contained in the set of eigenvalues of $\phi$ with respect to $\cA$, it suffices to show if $\lambda >M$ or $\lambda<m$, then $\lambda$ is not an eigenvalue.
  Accordingly suppose $\lambda>M$.
 In this case the measure of the set $S=\{z\in\mathbb T: \phi(z)>\lambda\}$ is less than $\frac{\pi}{2}$.
 On the other hand, if $\lambda$ is an eigenvalue, then there is a non-zero $c$ and outer function $h\in \HT$ such that
\[
  (\phi-\lambda)\norm{h}^2 = c\chi+(c\chi)^*
\]
 But then the measure of the set $S$ is $\frac{\pi}{2}$, a contradiction.
Which proves the corollary when $\lambda>M$.
 The proof of the case $\lambda<m$ proceeds analogously.
 It now follows that set of eigenvalues contains $(m,M)$ and is contained in $[m,M]$ and the proof of the corollary is complete.

To prove corollary \ref{cor:lambdatov} use \eqref{eq:NeilMagic} and the fact that $f_\lambda$ is outer to see that
\begin{align*}
f_{\lambda}(z)=&\exp\left(\int_{\MBBT}H(z,\cdot)\log\left(\frac{c\chi+(c\chi)^*}{\phi-\lambda}\right)^{1/2}\dmu\right)\\
 =&\exp\left(\int_{\MBBT}H(z,\cdot)\log\norm{c\chi+(c\chi)^*}^{1/2}\dmu\right)\exp\left(\int_{\MBBT}H(z,\cdot)\log\norm{\phi-\lambda}^{-1/2}\dmu\right)\\
 =&h_c(z) g_{\lambda}(z).
\end{align*}
Thus $f_{\lambda}(0)={h}_c(0) {g}_{\lambda}(0)$ and
$f^{\prime}_{\lambda}(0)= {h}_c(0) {g^\prime}_{\lambda}(0)+ {h^\prime}_c(0) {g}_{\lambda}(0)$
and the conclusion follows.

To prove Corollary \ref{cor:notAC} we will first show that the maps
\[
\begin{aligned}
\lambda &\mapsto g_{\lambda}(0) \text{ and}\\
\lambda &\mapsto g_{\lambda}^{\prime}(0)
\end{aligned}
\]
are twice differentiable with respect to $\lambda$ on $(m,M)$.
Those questions boil down to checking if
\begin{align*}
\lambda &\mapsto \int_\MBBT H(0,\cdot)\log\norm{\phi-\lambda}\dmu \text{ and}\\
\lambda &\mapsto \int_\MBBT H^{\prime}(0,\cdot)\log\norm{\phi-\lambda}\dmu
\end{align*}
are twice differentiable with respect to $\lambda$ on $(m,M)$.
 For a given  $\lambda_0\in(m,M)$ there is a  $\delta>0$ such that
 $\phi-\lambda$ is essentially bounded above and away from zero
  for $|\lambda -\lambda_0|<\delta.$ 
 It follows that for such $\lambda$, the functions $H(0,t)\log(|\phi(t)-\lambda|)$
 and $H^\prime(0)\log(|\phi(t)-\lambda|)$ as well as $(\phi(t)-\lambda)^{-1}$ 
 are all bounded above and below.
 Thus a standard application of the dominated convergence theorem establishes the desired differentiability. 
A similar argument shows that in fact both functions are infinitely differentiable. 
Since $\Phi_1\circ\Lambda$ is twice differentiable on $(m,M)$ it is locally Lipschitz.
Because $(m,M)$ can be written as a countable union of intervals with $\Phi_1\circ\Lambda$ Lipschitz on each interval the Hausdorff dimension of $\Phi_1\circ\Lambda((m,M))$ is at most $1$.
So $\Lambda((m,M))$ cannot be all of $\PCT\setminus\{[0,1]\}$ since $\Phi_1$ is injective on its range.
Let $\MCL=\left\{\MCV\ |\ \MCV \text{ is a nonzero proper subspace of }\MBBC+\MBBC z\right\}$.
Finally for each nonzero proper $\MCV$ choose a $f\in\HT_\MCV$ with $f(0)$ and $f^{\prime}(0)$ not both zero and let $\tau:\MCL\to\PCT$ be defined by map $\tau:\MCV\mapsto [f(0),f^{\prime}(0)]$.
The map $\tau$ is a bijection between $\MCL$ and $\PCT$, thus if $\tau(\MCV)\not\in\Lambda((m,M))$, then by \ref{cor:intNeil} and \ref{cor:lambdatov} we have that $T^\MCV_\phi$ has no eigenvalues.
Additionally if $m$ (resp. $M$) is an eigenvalue of $\phi$ relative to $\MCA$ then we need to add the condition that $\tau(\MCV)\neq\Lambda(m)$ (resp. $\tau(\MCV)\neq\Lambda(M)$) for the above conclusion to hold.

\section{Bundle shifts}
\label{sec:Boundaries}
  It is natural to ask what distinguishes the families of representations $\{\pi_\alpha : 0\le \alpha<1\}$ and $\{\pi_\MCV: \MCV \subset \MBBC \oplus \MBBC z\}$  of the algebras $\aA$ and $\MCA$ as multiplication operators on the spaces $\{ \HT_\alpha : 0\le \alpha <1\}$ and $\{\HT_\MCV : \MCV \subset \MBBC \oplus \MBBC z\}$ respectively.

  For the annulus, an answer is that Sarason recognized that the collection of representations $(\pi_\alpha,\HT_\alpha)$ played the same role on the annulus as the single representation determined by the shift operator $S$ given by
\begin{equation}
 \label{eq:repS}
  \AD \ni f\to f(S)
\end{equation}
  plays for the disc algebra $\AD$.
  For $\cA$ the representations $\pi_\MCV$ generate a family of positivity conditions sufficient for Pick interpolation in $\cA$ \cite{DPRS}.
  Likely it is a minimal set of conditions too.
  Corollary \ref{cor:neilbundleshifts} below can be interpreted as saying that the representations $(\pi_\MCV,\HT_\MCV; \MCV\subset \MBBC \oplus \MBBC z)$ should play the role of the rank one bundle shifts for the Neil algebra $\cA$.
In a dual direction \cite{DP} found a minimal set of test functions for $\MCA$. For similar results on multiply connected domains see \cite{BH1} and \cite{BH2}.

  For positive integers $n$, the algebra $M_n(B(H))$ of $n\times n$ matrices with entries from $B(H)$
  is naturally identified with $B(\mathbb C^n\otimes  H)$, the operators on the Hilbert space $\mathbb C^n\otimes H \equiv \oplus_1^n H$.
  In particular, it is then natural to give an element $X\in M_n(B(H))$ the norm $\|X\|_n$ it inherits as an operator
  on $\mathbb C^n\otimes H$.  If $A$ is a  subalgebra  of $B(H)$, then the norms $\|\cdot\|_n$ of course restrict to
  $M_n(A)$, the $n\times n$ matrices with entries from $A,$ and $A$ together with this sequence of norms is a concrete \df{operator algebra}.

    Turning to the disc algebra, an element $F\in M_n(\AD)$ takes the form $F=(F_{j,k})_{j,k=1}^n$ for $F_{j,k}\in \AD$.
  In particular, $M_n(\AD)$ is itself an algebra and comes naturally equipped with the norm,
\[
 \|F\|_n =\sup\{\|F(z)\|: z\in \mathbb D\}.
\]
  where $\|F(z)\|$ is the usual \df{operator norm} of the $n\times n$ matrix $F(z)$.
   The representation $\pi:\mathbb A(\mathbb D) \mapsto B(H^2)$ given by $\pi(a) = M_a$ extends naturally to $M_n(\AD)$ as $1_n\otimes \pi:M_n(\AD)\to B(\oplus^n H^2)$ by
\[
 1_n\otimes \pi (F) = \begin{pmatrix} \pi(F_{j,k})\end{pmatrix}_{j,k=1}^n.
\]
  Moreover, the maps $1_n\otimes \pi$ are isometric.
  Thus the algebra $\AD$ can be viewed as an operator algebra by identifying  $\AD$ together with the sequence of norms $(\|\cdot\|_n)$ with its image in $B(H^2)$ under the mappings $1_n\otimes \pi$.
	Of course, any subalgebra of $\AD$ can also then be viewed as an operator algebra by inclusion.

   Given an operator algebra $A$, a representation $\rho:A\to B(H)$ is \df{completely contractive} if $\|1_n\otimes\rho(F) \|_n \le \|F\|_n$ for each $n$ and $F\in M_n(A)$.
   If $A$ and $B$ are unital, then $\rho$ is a \df{unital representation} if $\rho(1)=1$.
   The representation $\rho$ on $B(H)$ is \df{pure} if
\[
   \bigcap_{a\in A} \rho(a) H = (0).
\]
  It is immediate that the representations of $\AD$ determined by $S$ as well as the representations  $\pi_\alpha$ of $\aA$ and $\pi_\MCV$ of $\cA$ are unital, completely contractive, and pure.

  Following Agler \cite{agler}, a completely contractive (unital) representation $\pi:\cA \to B(H)$ of $\cA$ on the Hilbert space $H$ is \df{extremal} if whenever $\rho:\cA\to B(K)$ is a completely contractive representation on the Hilbert space $K$ and $V:H\to K$ is an isometry such that
\[
   \pi(a) = V^* \rho(a)V
\]
 then in fact
\[
  V\pi(a) = \rho(a)V.
\]
  Given a Hilbert space $\MCN$, let $\HT_\MCN$ denote the Hilbert Hardy space of $\MCN$-valued analytic functions on the disc
 with square integrable boundary values. Associated to $\MCN$ is the representation $\rho:\AD\to B(\HT_\MCN)$ defined
 by
\[
 \rho(\varphi)f = \varphi f.
\]
 Thus, $\rho(\varphi)$ is multiplication by the scalar-valued $\varphi$ on the vector-valued $\HT$ space $\HT_\MCN$.
  Of course, $\HT_\MCN$ is naturally identified with $\MCN\otimes \HT$ and the representation $\rho$ is then the identity on $\MCN$ tensored with the representation of $\AD$ in Equation \eqref{eq:repS}.
   If $\rho$ is a  completely contractive unital pure extremal representation of $\AD,$  then there exists a Hilbert space $\MCN$ so that, up to unitary equivalence,   $\rho:\AD\to B(\HT_{\MCN})$ is
  given by $\rho(\varphi) f =\varphi f$.

  For $\MCA$ it turns out that the subspaces of $\HT_\MCN$ identified in \cite{mrinal}  give
  rise to the extremal representations. Indeed,
 given a Hilbert space $\MCN$ and a subspace $\MCV$ of the subspace $\MCN\oplus z\MCN$ of $\HT_\MCN$,
 the mapping $\pi_\MCV: \MCA \to B(\HT_\MCN \ominus \MCV)$ defined by
\[
  \pi_\MCV(a) f = a f,
\]
 is easily seen to be a unital pure completely contractive representation of $\MCA.$

\begin{theorem}
\label{thm:ExtremalChar}
 The representations $\pi_\MCV$ are unital pure completely contractive  extremal representations.
Moreover, if $\nu$ is a unital pure extremal completely contractive representation of $\MCA,$ then $\nu$ is unitarily equivalent to $\pi_{\MCV}$ for some Hilbert space $\MCN$ and $\MCV\subseteq \MCN \oplus \MCN z$.
\end{theorem}

Finally we will say a representation has rank one if there does not exist a nontrivial orthogonal pair of subspaces invariant for the representation. 

\begin{corollary}
 \label{cor:neilbundleshifts}
    The representations $\pi_\MCV$ for $\MCV \subset \MBBC \oplus \MBBC z$ have rank one.
    Moreover if the representation  $\pi$ is a unital pure extremal completely contractive rank one representation of $\MCA$, then there is a $\MCV\subset \MBBC \oplus \MBBC z$ such that $\pi$ is unitarily equivalent to $\pi_\MCV$.
\end{corollary}

The remainder of the section is organized as follows. Subsection \ref{subsec:proofExtremal} proves that the representations $\pi_\MCV$ are extremal. Subsection \ref{subsec:proofExtremalChar} contains the proof of the remainder of Theorem \ref{thm:ExtremalChar}. The corollary is proved in \ref{subsec:proofCorExtremal}.

\subsection{The Extremal Representations of $\MCA$}
\label{subsec:proofExtremal}

While it is easy to see that the representations $\pi_\MCV:\MCA \to B(\HT_\MCV)$ are unital, pure, and completely contractive showing that they are also extremal is a bit harder. To prove they are extremal we will first prove a proposition which gives us an easy to verify sufficient condition for a representation to be extremal.

The first lemma we need is a well known  generalization of Sarason's Lemma \cite{Sarason-ss}.
 Given a representation $\rho:A\to B(K)$, a subspace $\MCM$ of $K$ is \df{invariant} for $\rho$
  if $\rho(a)\MCM \subset \MCM$ for all $a\in A$.  A subspace $H$ of $K$ is \df{semi-invariant}
  for $\rho$ if there exist invariant subspaces $\MCM$ and $\MCN$ such that $H= \MCN\ominus \MCM$.
  Note that, letting $V:H\to K$ denote the inclusion, the mapping $A\ni a\mapsto V^* \rho(a) V$ is also a representation of $A$.

 \begin{lemma}
 \label{lem:Sarason}
Let $\nu:A\to B(H)$ be a representation of $A$ in $B(H)$ and $\rho:A\to B(K)$ be a representation of $A$ in $B(K)$ and $V:H\to K$ an isometry.
If $\nu(a)=V^*\rho(a)V$ for all $a\in A$, then $VH$ is a semi-invariant for $\rho$.
\end{lemma}

\begin{proof}
Let
  \[
    \MCN\coloneqq\bigvee_{a\in A}\rho(a)VH,
  \]
 the smallest (closed) subspace of $K$ containing all of the spaces $\rho(a)VH$.
Notice that the elements of the form
  \[
    \sum_{i=0}^{N}\rho(a_i)Vh_i \text{ where }\{h_i\}\subset H,\ \{a_i\}\subset A,\text{ and } N>0
  \]
form a dense subset of $\MCN$.
Since $\rho$ is a representation, for any $a\in A$, $\{h_i\}\subset H$, $\{a_i\}\subset A$,and $N>0$ we have that
  \[
    \rho(a)\left(\sum_{i=0}^{N}\rho(a_i)Vh_i\right)=\sum_{i=0}^{N}\rho(a\cdot a_i)Vh_i\in\MCN
  \]
and thus $\MCN$ is $\rho(a)$ invariant.
To complete this direction of the proof we only need to show that $\MCM\coloneqq\MCN\ominus VH$ is also invariant for $\rho(a)$.
Notice that $\nu(a)=V^*\rho(a)V$ implies
  \begin{align*}
    V^*\rho(a)\left(\sum_{i=0}^{N}\rho(a_i)Vh_i\right)&=\sum_{i=0}^{N}V^*\rho(a \cdot a_i)Vh_i\\
    &= \sum_{i=0}^N \nu(a a_j)h_j \\
    &=\sum_{i=0}^{N}\nu(a)\nu(a_i)h_i\\
    &=\nu(a)\sum_{i=0}^{N}V^*\rho(a_i)Vh_i.
  \end{align*}
Thus $V^*\rho(a)|_{\MCN}=\nu(a)V^*|_{\MCN}$.
If  $m\in\MCM,$ then $m\in\MCN$ and by the Fredholm alternative  $V^*m=0$.
Thus, if  $a\in A,$ then $V^*\rho(a)m=\nu(a)V^*m=0$, which, again by the Fredholm alternative, implies $\rho(a)m\in\MCM$.
\end{proof}

The next lemma allows us to improve semi-invariance to invariance if $\nu(a)$ is an isometry and $\Norm{\rho(a)}=1$.
\begin{lemma}
\label{lem:invariant}
If $H\subset K$ is a semi-invariant subspace for a contraction $T$ and $S\coloneqq P_H T|_H$ is an isometry, then $H$ is an invariant subspace for $T$.
\end{lemma}

\begin{proof}
Since $H$ is semi-invariant for $T$ and $S=P_H T|_H$ we know that there exists two $T$ invariant spaces $\MCN$ and $\MCM$ such that $\MCN=\MCN\oplus H$ and
\[
	T=
	\begin{bmatrix}
	A & B & C \\
	0 & S & F \\
	0 & 0 & K
	\end{bmatrix}
\]
Where $A:\MCM \to \MCM$, $B:\MCM \to H$, $C:\MCM \to \MCN^{\perp}$, $F:H \to \MCN^{\perp}$, and
 $K:\MCN^{\perp} \to \MCN^{\perp}$.
Since $T$ is a contraction we have $I-T^*T \geq 0$ thus, for all $h\in H$, 
\[
	0\leq \innerP{(I-T^*T)
	\begin{bmatrix}
	0 \\
	h \\
	0
	\end{bmatrix}
	}{
	\begin{bmatrix}
	0 \\
	h \\
	0
	\end{bmatrix}
	}
	=\innerP{\begin{bmatrix}
	-A^*Bh \\
	\left(I-B^*B-S^*S\right)h \\
	-\left(C^*B+F^*S\right)h
	\end{bmatrix}
	}{
	\begin{bmatrix}
	0 \\
	h \\
	0
	\end{bmatrix}
	}
	=\Norm{h}^2-\Norm{Bh}^2-\Norm{Sh}^2=-\Norm{Bh}^2
\]
so $Bh=0$ for all $h \in H$.
Thus $H$ is a $T$ invariant subspace.
\end{proof}

Combining Lemmas \ref{lem:Sarason} and \ref{lem:invariant} yields the following proposition.
\begin{proposition}
\label{prop:extremal}
Let $\nu:A\to B(H)$ be a contractive representation of $A$ in $B(H)$ and $\{a_i\}_{i\in J}\subset A$ be a set that generates a dense subalgebra of $A$ with $\Norm{a_i}=1$ for all $i\in J$.
If $\nu(a_i)$ is an isometry for all $i\in J$, then $\nu$ is extremal.
\end{proposition}

\begin{proof}
Let $\rho:A\to B(K)$ be a contractive representation of $A$ in $B(K)$ and $V:H\to K$ be an isometry such that $\nu(a)=V^*\rho(a)V$ for all $a\in A$.
By lemma \ref{lem:Sarason},  $VH$ is a semi-invariant subspace of $K$ for $\rho$.
By lemma \ref{lem:invariant},  $VH$ is invariant for $\rho(a_i)$ for each $i \in J$.
Because $\rho$ is a representation we have that $VH$ is invariant for $\rho(a)$ for each $a$ in the algebra generated by the set $\{a_i\}_{i\in J}$. Thus $VH$ is invariant
 for $\rho$. 

Now since $VH$ is $\rho(a)$ invariant and $\nu(a)=V^*\rho(a)V$ for all $a\in A$ we have that
\[
	V\nu(a)=VV^*\rho(a)V=P_{VH}\rho(a)V=\rho(a)V \text{ for all }a\in A.
\]
Thus $\nu$ is extremal.
\end{proof}

Now it is easy to show that all of the $\pi_\MCV$'s are extremal representations of $\MCA$.

\begin{corollary}
  The representation $\pi_\MCV$ is an extremal representation of $\MCA$.
\end{corollary}

\begin{proof}
  Since $\pi_{\MCV}(1)$, $\pi_{\MCV}(z^2)$, and $\pi_{\MCV}(z^3)$ are isometries and $1$, $z^2$, and $z^3$ generate $\MCA$, by proposition \ref{prop:extremal} we know that $\pi_{\MCV}$ is extremal.
\end{proof}

\subsection{Proof of Theorem \ref{thm:ExtremalChar}}
\label{subsec:proofExtremalChar}

Let $\nu:\MCA\to B(H)$ be a pure extremal representation of $\MCA$ on some separable Hilbert space $H$.
  By \cite[Corollary 7.7]{P}, the representation $\nu$ has an $C(\MBBT)$-dilation; i.e., there exists a completely contractive representation $\rho:\LINF(\MBBD)\to B(K)$ and an isometry $V:H\to K$ such that $\nu(a) = V^* \rho(a)V$ for all $a\in \MCA$.
   Moreover since $\nu$ is extremal $V\nu(a)=\rho(a)V$ for all $a\in\MCA$ and $VH$ is invariant for $\rho$. 
Finally let
\[
E=\bigvee_{i=0}^{\infty}\rho(z^i)VH\subset K.
\]
 Since $z^i\in\MCA$ for all $i\in\MBBN$ and $i\neq1,$
  \[
    \bigvee_{\substack{i=0\\i\neq1}}^{\infty}\rho(z^i)VH=\bigvee_{\substack{i=0\\i\neq1}}^{\infty}V\nu(z^i)H=VH.
  \]
 In particular,  $E=\rho(z)VH \vee VH$.

First we will show that $S=\rho(z)|_{E}$ is a pure isometry on $E$; if $f,g\in K$, then
\[
  \innerP{\rho(z)f}{\rho(z)g}=\innerP{\rho(z)^*\rho(z)f}{g}=\innerP{\rho(\overline{z}z)f}{g}
  =\innerP{f}{g}.
\]
Since $S$ is the restriction of an isometry to an invariant subspace $S$ is an isometry.
To show that $S$ is pure note that
  \[
    \rho(z^2)E=\rho(z^2)\left(\rho(z)VH\vee VH\right)=\rho(z^3)VH\vee\rho(z^2)VH\subset VH.
  \]
Since $\nu$ is pure we have
\begin{align*}
  \bigcap_{b\in\AD} \rho(b)E &\subset\bigcap_{\substack{b=z^2a\\a\in\MCA}} \rho(a) \rho(z^2)E\\
  &\subset\bigcap_{a\in\MCA}\rho(a)VH\\
  &=\bigcap_{a\in\MCA}V\nu(a)H=\{0\}.
\end{align*}
Thus $S$ is a pure shift on $E$.

 Since $S$ is a pure shift there is a Hilbert space $\MCN$ and a unitary map $W:E \to \HT_{\MCN}$ such that $W S=M_z W$.
 Since  the subspace $S^2 E$ lies in $VH$ and $VH$ is a subspace of $E$, there exists a subspace $\MCV$ of $\MCN \oplus z \MCN$ such that $WVH = \HT_{\MCV} =\HT_\MCN \ominus \MCV$.
 Let $U:H \to \HT_{\MCV}$ be defined by $U=WV$, this is a unitary map such that
 \[
  U^*\pi_{\MCV}(a)Uh=U^*M_aUh=\nu(a)h \text{ for all } h\in H,
 \]
 i.e. $\pi_{\MCV}$ is unitarily equivalent to $\nu$.

\subsection{Proof of Corollary \ref{cor:neilbundleshifts}}
\label{subsec:proofCorExtremal}
  Suppose $\MCV\subset \MBBC \oplus \MBBC z$ and $M$ and $N$ are orthogonal subspaces of $\HT_\MCV$ invariant for $\MCA$.
  Choosing non-zero   $\varphi$ and $\psi$ from $M$ and $N$ respectively, it follows that
  $\langle z^m \varphi, z^n\psi\rangle =0$ for natural numbers $m\ne 1\ne n$.
  Hence if $\mu$ is normalized arclength measure on $\MBBT$ and $\chi(z)=z$,
  \[
    0 =\int_\MBBT \varphi \overline{\psi} \chi^j \dmu
  \]
  for all $j$ and therefore $\varphi \overline{\psi}=0$.
  Since both $\varphi$ and $\psi$ are in $\HT$, each is non-zero almost everywhere whenever it is not the zero function.
  Thus at least one must be zero, which is a contradiction.
  So if $\MCV\subset \MBBC \oplus \MBBC z,$ then $\pi_\MCV$ is rank one.

  By theorem \ref{thm:ExtremalChar} it suffices to check the second part of the corollary for $\pi_\MCV$ where $\MCV\subset \MCN \oplus \MCN z$.
  If $\MCN$ is one dimensional then $\MCN$ is unitarily equivalent to $\MBBC$ and we are done.
  If $\MCN$ is not one dimensional, then choose a pair of non-zero vectors $e$ and $f$ in $\MCN$ such that $\innerP{e}{f}=0$ and let $\MCE = z^2 \HT e$ and $\MCF= z^2 H^2 f$.

  Both $\MCE$ and $\MCF$ are non-trivial subspaces of $\HT_\MCV$ for any $\MCV$ and are $\MCA$ invariant.
  They are also orthogonal by construction. Hence $\pi_\MCV$ is not rank one.

\end{document}